\documentclass[a4paper, 11pt]{article}

\usepackage{amsmath,amssymb,amsthm}
\usepackage[latin1]{inputenc}
\usepackage{version,tabularx,multicol}
\usepackage{graphicx,float,psfrag}
\usepackage{stmaryrd}
 \usepackage{enumerate}
 \usepackage{color}
\usepackage[pdftex]{hyperref}
\usepackage{authblk}

\usepackage[numbers,sort&compress]{natbib}

\theoremstyle{plain}
\newtheorem{theorem}{Theorem}[section]

 \newtheorem{lemma}[theorem]{Lemma}

\newtheorem{remark}{Remark}[section]

  \makeatletter
  \@addtoreset{equation}{section}
  \makeatother

 \def\beqlb{\begin{eqnarray}}\def\eeqlb{\end{eqnarray}}
 \def\beqnn{\begin{eqnarray*}}\def\eeqnn{\end{eqnarray*}}

 \def\qed{\hfill$\Box$\medskip}

\newcommand{\bcen}{\begin{center}}
\newcommand{\ecen}{\end{center}}
\newcommand{\bgeqn}{\begin{equation}}
\newcommand{\edeqn}{\end{equation}}

\begin{document}

\title{Small positive values and limit theorems for supercritical branching processes with immigration in
random environment}

\author{ Yinxuan Zhao\thanks{%
School of Mathematical Sciences \& Laboratory of Mathematics and Complex
Systems, Beijing Normal University, Beijing 100875, P.R. China. Email:
yinxuanzhao@mail.bnu.edu.cn} and Mei Zhang\thanks{%
Corresponding
author. School of Mathematical Sciences \& Laboratory of Mathematics and Complex
Systems, Beijing Normal University, Beijing 100875, P.R. China. Email:
meizhang@bnu.edu.cn} }
\maketitle

\noindent{\bf Abstract}\quad
 Let $(Z_n)$ be a supercritical branching process with immigration in a random environment. The small positive values and some lower deviation inequalities for $Z$  are investigated. Based on these results, the central limit theorem of $\log Z_n$ and the Edgeworth expansion are obtained. The study is taken under the assumption that each individual produces $0$ offspring with a positive probability.
 \vspace{0.3cm}



\bigskip

\section{Introduction}
 Branching processes in a random environment (BPRE) are popular models introduced by Smith and Wilkinson \cite{smith1969} and have been extensively studied in recent years, see \cite{Athreya1971a, Athreya1971b}, \cite{2013ban}-\cite{2023grama}, \cite{2012huang, Vat2017}, etc. As a nature extension of BPRE and due to its interest in applications, the branching processes with immigration in a random environment (BPIRE) attracted attention of many researchers. For example,  Kesten et al~\cite{kks1975}, Key \cite{key1987}, Hong and Wang \cite{HW2016} used BPIRE to get asymptotics of a random walk in a random environment; Bansaye \cite{Ban2009} investigated BPIRE to study cell contamination; Vatutin \cite{vat2011} considered BPIRE to study polling systems with random regimes of service. Recently,  %Wang and Liu~\cite{wang2017, 2021liu} obtained some limit theorems for supercritical BPIRE and Huang et al~\cite{Huang2022} studied its moments and large deviations. %For the study of the critical and subcritical BPIRE, see \cite{2021afa, 2021vat}.%For the Galton-Watson process with and without immigration, can be found in \cite{1972Athreya}, \cite{1971pakes}, \cite{2016lz} and  \cite{2017sz}.
 for classical BPRE $(Z_n^0)$, Buraczewski and Damek \cite{2022dar} proved the central limit theorem of $\log Z_n^0$ conditionally on the survival set, which improved the result in \cite{2012huang} where the hypothesis $\mathbb{P}(Z_1^0=0)=0$ is needed.  A number of results including local probability estimates, central limit theorems, Berry-Esseen estimates, Cram\'{e}r's large deviation expansion, Edgeworth expansion and renewal theorems have been recently proved in \cite{2014ban}--\cite{2023grama}. For BPIRE $(Z_n)$, the rate of $\mathbb{P}(Z_n=j|Z_0=k)$ under the assumption $\mathbb{P}(Z_1=0)=0$ was obtained in \cite{Huang2022}.  Wang and Liu~\cite{wang2017, 2021liu} proved the central limit theorem of $\log Z_n$ and  got the Berry-Esseen bound under the assumption that $\mathbb{P}(Z_1=0)=0$.

  In the paper, we focus on the supercritical BPIRE $(Z_n)$ in the case that $\mathbb{P}(Z_1=0)>0$.
  We start with the decay rate of $n$-step transition probability $\mathbb{P}(Z_n=j|Z_0=k)$ and show  its exponential rate  as $n\rightarrow\infty$ by Theorem \ref{small} in Section 3. Based on this, together with the discussion on the path of $Z_n$, we give the  lower-deviation-type inequality  by Theorem \ref{lowerdeviation}.
 Differently from \cite{Huang2022},  we use the argument related with Markov chain to simplify the estimation of $\mathbb{P}(Z_n=j|Z_0=k)$ (Lemma \ref{smallequal}),  then consider about the decomposition of $Z_n$ by tracing the ancestry of the individual in generation $n$ (Lemma \ref{<0}), and finally use the property of associated random walk to get the proof. In Section 4, under the assumption $\mathbb{P}(Z_1=0)>0$,  we obtain the central limit theorem of $\log Z_n$, the Edgeworth expansion and the renewal theorem   (Theorems \ref{CLT}-- \ref{renewal}). We first estimate the harmonic moments of $Z_n$ and the moments of $\log Z_n$ by the lower deviation of $Z_n$, then study the deviation between $Z_{n+1}$ and $m_{n+1}Z_n$, and consequently get the asymptotic behavior of the Fourier transforms of $\log Z_n$ which yields the desired results. 

The rest of this paper is organized as follows. In Section 2, we present the description of the model and basic assumptions. In Section 3, we prove the small positive values and lower deviation of $Z_n$.  The limit theorems of $\log Z_n$ are studied in Section 4. 
In the following context,  $C,C_1,C_2,\cdots,\beta,\beta_0,\beta_1,\beta_2,\cdots$ denote positive constants whose value may change from place to place. With $f(n)=o(g(n))$ as $n\rightarrow \infty$, we refer that $\lim_{n\to \infty} f(n)/g(n)=0$. With $f(n)=O(g(n))$ as $n\rightarrow \infty$, we refer that there exists  $M>0$ such that $\limsup_{n\rightarrow \infty}f(n)/g(n)\leq M$.

\section{Description of the model}
We now give a description of the model. Let $\Delta =\left( \Delta _{1},\Delta _{2}\right) $ be the space of all
pairs of probability measures on   $\mathbb{N}=\{0,1,2,\ldots \}.$
Equipped with the component-wise metric of total variation $\Delta $ becomes
a Polish space. Let $\mathbf{Q}=\{f,h\}$ be a random vector with  independent 
 components taking values in $\Delta $, and let $\mathbf{Q}%
_{n}=\{f_{n},h_{n}\},n=1,2,\ldots ,$ be a sequence of independent copies of $%
\mathbf{Q}$. The infinite sequence $\xi=\left\{ \mathbf{Q}_{1},%
\mathbf{Q}_{2},...\right\} $ is called a random environment.  We denote by $\mathbb{P}(\cdot|\xi)$ the $quenched$ $probability$, i.e., the conditional probability when the environment $\xi$ is given and define the $annealed$ $probability$ by $\mathbb{P}(\cdot)=\mathbb{E}[\mathbb{P}(\cdot|\xi)]$. %The corresponding quenched and annealed expectations are denoted respectively by $\mathbf{E}_\xi$ and $\mathbb{E}$.
A sequence of $\mathbb{N}$-valued random variables $\mathbf{Z}=\left\{
Z_{n},\ n\in \mathbb{N}\right\} $ specified on the
probability space $(\Omega,\mathcal{F},\mathbb{P})$ is called a branching
process with immigration in a random environment, if $Z_{0}=k\in\mathbb{N}\backslash\{0\}$ and, given $\xi$ the process $\mathbf{Z}
$ is a Markov chain with
\begin{equation}\label{define}
\mathcal{L}\left( Z_{n}|Z_{n-1}=z_{n-1},\xi=(\mathbf{q}_{1},\mathbf{q%
}_{2},...)\right) =\mathcal{L}(N _{n,1}+\ldots +N_{n,z_{n-1}}+Y_{n})
\end{equation}%
for every $n\in \mathbb{N}\backslash \left\{ 0\right\} $, $%
z_{n-1}\in \mathbb{N}$ and $\mathbf{q}_{1}=\left( F_{1},H_{1}\right),%
\mathbf{q}_{2}=\left( F_{2},H_{2}\right),\cdots\in \Delta$, where $N
_{n,1},N_{n,2},\ldots $ are i.i.d. random variables with distribution $%
F_{n} $ and independent of the random variable $Y_{n}$ with distribution
$H_{n} $. In the language of branching processes, $Z_{n-1}$ is the $(n-1)$th
generation size of the population, $F_{n}$ is  the offspring distribution
of each individual at generation $n-1$ and $H_{n}$ is the
law of the number of immigrants  at generation $n$. Note that we do not assume the independence between the random distributions $f_n$ and $h_n$ for fixed $n$.

 Along with the process $\mathbf{Z}$, we consider the classic branching process $\mathbf{%
Z^0}=\left\{Z^0_{n},\ n\in \mathbb{N}\right\} $ in the random environment $%
\xi$. Given $\xi$, $\mathbf{%
Z^0}$ is a Markov chain with  $Z^0_{0}=Z_0$ and, for $n\in \mathbb{N}_0$,
\begin{equation}\label{classic}
\mathcal{L}\left( Z_{n}^0\bigg|Z_{n-1}^0=z_{n-1},\xi=(\mathbf{q}_{1},\mathbf{q%
}_{2},...)\right) =\mathcal{L}(N_{n,1}+\ldots +N_{n,z_{n-1}}).
\end{equation}

Consider the
so-called associated random walk $\mathbf{S}=\left( S_{0},S_{1},...\right) $%
. This random walk has initial state $S_{0}=0$ and increments $%
X_{n}=S_{n}-S_{n-1}$, $n\geq 1$, defined as
$
X_{n}:=\log \mathfrak{m}\left( f_{n}\right),
$
which are i.i.d. copies of  $X:=\log $ $%
\mathfrak{m}(f)$ with
$
\mathfrak{m}(f):=\sum_{j=0}^{\infty }jf\left( \left\{ j\right\} \right).
$ With each pair of measures $(f,h)$ we associate the respective probability
generating functions (p.g.f.)
\begin{equation*}
f(s):=\sum_{j=0}^{\infty }f\left( \left\{ j\right\} \right) s^{j},\qquad
h(s):=\sum_{j=0}^{\infty }h\left( \left\{ j\right\} \right) s^{j}.
\end{equation*}
For convenience, denote $m_n:=\mathfrak{m}\left( f_{n}\right)$ and $\lambda_n:=\sum_{j=0}^{\infty }jh_n\left( \left\{ j\right\} \right) $. Then $m_n=\mathbb{E}[N_{n,i}|\xi]$ and $\lambda_n=\mathbb{E}[Y_n|\xi]$.
 For $0\leq m\leq n$, let $f_{m,n}$ be the convolutions of the probability generating functions $f_1,\cdots,f_n$ specified  by
\beqnn
f_{m,n}:=f_{m+1}\circ\cdots\circ f_n
\eeqnn
 with $f_{n,n}(s):=s$ by convention.

In the following context, we always assume that
 \beqnn
  0<m_1<\infty \quad  \mbox{a.s.}%\quad  \mbox{and}\quad
%\mathbb{E}\lambda_1<\infty.
\eeqnn
We consider the process under the condition
 $$\mathbb{E}\log m_1\in (0,\infty),$$ which means the process $\mathbf{%
Z^0,Z}$ is supercritical.

  Throughout the paper, we study the model under the assumption that $\mathbb{P}(h(0)<1)>0$. When $\mathbb{P}(h(0)<1)=0$, our model degenerates to BPRE, then the main results below (Theorems \ref{small}-- \ref{lowerdeviation}, \ref{CLT}--\ref{renewal}) coincide with \cite[Theorem 2.1]{2014ban} and \cite[Lemma 3.1, Theorems 2.2--2.4]{2022dar}. 

%Define the  filtration
%$
%\mathcal{F}_n:=\sigma(\xi,Y_k,N_{k,i}:1\leq k\leq n,i\geq 1)
%$
%with $\mathcal{F}_0:=\sigma(\xi)$ by convention.
We use $\mathbb{E}_k$ and $\mathbb{P}_k$  to denote the expectation and probability, respectively,   emphasizing  the process with $k$ initial individuals, i.e., $\mathbb{P}_k(\cdot):=\mathbb{P}(\cdot|Z_0=k)$.

\section{Small positive values and lower deviation}
\subsection{Basic assumptions and main results}
We need the following assumptions:

 \noindent\textbf{Assumption (A)}\quad  $\mathbb{P}(h(0)>0,f(0)>0,f(\{1\})>0)>0$. 

%\noindent\textbf{Assumption (A)}\quad  $\mathbb{E}[h(\{1\})f(\{1\})]>0$ and $h(0)>\gamma$, $a.s.$ for some $\gamma\in(0,1)$.

\noindent\textbf{Assumption (B)}\quad There exists $ \delta\in (0,1)$ such that $f(\{0\})<\delta$, $a.s.$
%\noindent\textbf{Assumption (B)}\quad There exists $ \delta\in (0,1)$ such that $f(\{0\})\in[1-\delta,\delta]$ $a.s.$

 \begin{remark}\label{hutong}
	We observe that under \textbf{Assumption (A)}, for all $k,j\geq 1$, there exists $l\geq 0$ such that $\mathbb{P}_k(Z_l=j)>0$.
\end{remark}

In this section, we state our limit theorems for small positive values and lower deviation of $Z_n$. We start with the exponential rate of $\mathbb{P}_k(Z_n=j)$ under our assumption:

\begin{theorem}\label{small}
Assume that   \textbf{Assumptions  (A) (B)} hold. Then there exists $\varrho\in(-\infty,0)$ such that  for all $k,j\geq 1$,
	\beqnn
	 \lim_{n\rightarrow\infty}\frac{1}{n}\log\mathbb{P}_k(Z_n=j):=\varrho.
	\eeqnn
Moreover, for every sequence $\{k_n\}$ such that $k_n\geq 1$ for $n$ large enough and $k_n/n\rightarrow 0$ as $n\rightarrow\infty$,
	\beqnn
	\lim_{n\rightarrow\infty}\frac{1}{n}\log\mathbb{P}_1(1\leq Z_n\leq k_n)=\varrho.
	\eeqnn

\end{theorem}

%{\blue For the Galton-Watson process with or without immigration, the similar result is well-known and can be found in \cite{1972Athreya},  \cite{1971pakes}, \cite{2016lz} and  \cite{2017sz}.}

For classical BPRE $\mathbf{Z^0}$, Bansaye and B\"{o}inghoff \cite{2014ban} proved that under assumption $\mathbb{E}\log m_1>0$ and $\mathbb{P}(f(0)>0)>0$, for all $k,j\in Cl(\mathcal{I})$,
\beqnn
\lim_{n\rightarrow\infty}\frac{1}{n}\log\mathbb{P}_k(Z^0_n=j)=\rho\in(-\infty,0],
\eeqnn
where $\mathcal{I}:=\{j\geq 1:\mathbb{P}(f(\{j\})>0,f(0)>0)>0\}$ and $Cl(\mathcal{I}):=\{k\geq 1:\exists n\geq 0\ \mbox{and}\ j\in\mathcal{I} \ \mbox{with}\ \mathbb{P}_j(Z^0_n=k)>0\}$;  If further assume $f(0)<\delta$, $a.s.$ with some $\delta\in(0,1)$ and $\mathbb{E}\log m_1<\infty$, then $\rho<0$. Moreover, if $f$ is fractional linear, the constant $\rho$ was computed in \cite{2014ban} and \cite{2014boi}.
 In the case $f(0)=0$, a.s.,   Grama et al \cite[Theorem 2.4]{2023grama} proved that under assumption $\mathbb{P}(0<f(\{1\})<1)>0$, for all state $j\in Cl(\{k\})$,
\beqnn
\mathbb{P}_k(Z_n^0=j)\sim q_{k,j}(\mathbb{E}[f(\{1\})^k])^n
\eeqnn
as $n\rightarrow\infty$ with some constants $q_{k,j}\in(0,\infty)$.

For BPIRE $\mathbf{Z}$, in the case  $f(0)=0, a.s.$, it has been proved in \cite[Theorem 1.2]{Huang2022} that if $\mathbb{E}[h(0)f(\{1\})^k]>0$, then for all state $j\in Cl(\{k\})$,
\beqnn
\mathbb{P}_k(Z_n=j)\sim \bar{q}_{k,j}(\mathbb{E}[h(0)f(\{1\})^k])^n
\eeqnn
for some constants $\bar{q}_{k,j}>0$.  % It has been stated mistakenly in \cite{Huang2022} that $\bar{q}_{k,j}<\infty$ without any other assumptions, whereas the extra condition $\mathbb{P}(0<f(\{1\})<1)>0$ is needed to ensure that $\gamma_k-\gamma_{m+1}>0$  in the last inequality of the proof of \cite[Theorem 1.2(a)]{Huang2022}.
%In the present paper,  we assume $\mathbb{P}(h(0)>0,f(0)>0,f(\{1\})>0)>0$.

Our next result is the lower-deviation-type inequality of $Z_n$:

\begin{theorem}
	\label{lowerdeviation}
	
	Assume that \textbf{Assumptions  (A) (B)} hold. Then there are   $\theta\in (0,\mathbb{E}\log m_1),\beta>0$ and $C>0$ such that for all $k\geq 1$,
	\beqlb\label{lowerde}
	\mathbb{P}_k(1\leq Z_n\leq e^{\theta n})\leq Ce^{-\beta n}.
	\eeqlb
\end{theorem}

%In the case $f(0)=0$, a.s., more precise property about the lower deviations of $\mathbf{Z}$ was shown in \cite{Huang2022}.For classical BPRE, there are many papers concerning the lower deviation of $\mathbf{Z^0}$, see, for example  \cite{2013ban} and \cite{2017grama}.

\subsection{Proofs}
  Recall the definition (\ref{define}). {For $k\geq 1$, let
	\beqnn
g_n(k,s):=\mathbb{E}_k[s^{Z_n}|\xi],\quad s\in[0,1]
	\eeqnn
	be the probability generating function of $Z_n$ under the quenched law. It is obvious that
\beqlb\label{quenchedpdf}
	g_n(k,s)=(f_{0,n}(s))^k\prod_{i=1}^n h_i(f_{i,n}(s)).
	\eeqlb

  Firstly, we give the crucial argument for exponential upper bound of $\mathbb{P}_k(Z_n=0)$. For this purpose, we need a technical lemma borrowing from the proof of \cite[Proposition 2.2(i)]{2014ban}:

\begin{lemma}\label{nuj-estimation}
  Assume that \textbf{Assumption (B)} hold.  Let $a\in \mathbb{N}$
be fixed and introduce
$$ \bar{f_1}(\{j\}) := f_1(\{j\}),\quad 0\le j< a, \quad \bar{f_1}(\{a\}) = f_1([a,\infty)).
$$
The corresponding truncated random variables are denoted similarly, e.g. by $\bar{X}$ and $\bar{S}$. Define $\breve{S}_n=\bar{S}_n-\varepsilon n$ with $0<\varepsilon<\mathbb{E}(\bar{X})$.  Define the prospective minima of $\breve{S}$ which are defined by $\nu(0):=0$ and
$$
\nu(j):=\inf\{n > \nu(j-1): \breve{S}_k>\breve{S}_n, \forall k>n\}.
$$
Then there exists $d\in(0,1)$ such that for all $j\geq 1$, $f_{\nu(j),n}(0)\leq 1-d$,  a.s. Moreover,  there exists $0<\epsilon< \mathbb{E}[\nu(1)]^{-1}$ such that for $n\geq 1$, 
 \beqnn
	\mathbb{P}(\sharp\{j\geq 0:\nu(j)\leq n\}<\epsilon n)\leq e^{-\alpha n}
	\eeqnn
 with some $\alpha>0$, where $\sharp\{\cdot\}$ denotes the cardinality of  the set.
\end{lemma}

\begin{lemma}\label{pzn=0}
	Assume that  \textbf{Assumptions (B)} hold. Then there are constants $\beta, C>0$ such that for all $k\geq 1$,
	\beqnn
	\mathbb{P}_k(Z_n=0)\leq Ce^{-\beta n}.
	\eeqnn
\end{lemma}

\begin{proof}
	From (\ref{quenchedpdf}) we have
	\beqlb\label{kzn=0}
	\mathbb{P}_k(Z_n=0)=\mathbb{E}[g_n(k,0)]=\mathbb{E}\bigg[(f_{0,n}(0))^k\prod_{i=1}^n h_i(f_{i,n}(0))\bigg]\leq \mathbb{E}\bigg[\prod_{i=1}^n h_i(f_{i,n}(0))\bigg].
	\eeqlb
	By Lemma \ref{nuj-estimation},
	\beqlb\label{secpart}
	\mathbb{E}\bigg[\prod_{i=1}^nh_i(f_{i,n}(0))\bigg]&\leq & \mathbb{P}(\sharp\{j\geq 0:\nu(j)\leq n\}<\epsilon n)+\mathbb{E}\bigg[\prod_{i=1}^nh_i(f_{i,n}(0));\sharp\{j\geq 0:\nu(j)\leq n\}\geq\epsilon n\bigg]\crcr
	&\leq & e^{-\alpha n}+(\mathbb{E}[h(1-d)])^{\epsilon n-1}\crcr
	&\leq & Ce^{-\beta n}
	\eeqlb
	with some constant $C,\beta>0$. Combining (\ref{kzn=0}) with (\ref{secpart}) yields the desired result.
\end{proof}

\begin{lemma}\label{smallequal}
	Assume that  \textbf{Assumptions  (A)} hold. Then for all $k,j\geq 1$, it holds that
	\beqnn
	\lim_{n\rightarrow\infty}\frac{1}{n}\log\mathbb{P}_k(Z_n=j)=\lim_{n\rightarrow\infty}\frac{1}{n}\log\mathbb{P}_1(Z_n=1)\in (-\infty,0].
	\eeqnn
	Moreover, for every sequence $\{k_n\}$ such that $k_n\geq 1$ for $n$ large enough and $k_n/n\rightarrow 0$ as $n\rightarrow\infty$,
	\beqnn
	\lim_{n\rightarrow\infty}\frac{1}{n}\log\mathbb{P}_1(Z_n=1)=\lim_{n\rightarrow\infty}\frac{1}{n}\log\mathbb{P}_1(1\leq Z_n\leq k_n).
	\eeqnn
\end{lemma}

\begin{proof}
	We use  similar methods as in \cite[Lemma 4.1]{2014ban}.
	 Note that $\mathbb{P}_k(Z_1=1)\geq \mathbb{P}(f(0)^{k-1}f(\{1\}))>0$ for all $k\geq 1$ under Assumption (A). 
	By Markov property, for all $m,n\geq 1$,
	\beqnn
	\mathbb{P}_1(Z_{n+m}=1)\geq\mathbb{P}_1(Z_n=1)\mathbb{P}_1(Z_m=1).
	\eeqnn
	Adding that $\mathbb{P}_1(Z_1=1)>0$, we obtain that the sequence $-\log\mathbb{P}_1(Z_n=1)$ is finite and subadditive. Then $\lim_{n\rightarrow\infty}\frac{1}{n}\log\mathbb{P}_1(Z_n=1)$ exists and belongs to $(-\infty,0]$.
  By Remark \ref{hutong}, for all $k,j\geq 1$, there exist $l,m\geq 0$ such that $\mathbb{P}_1(Z_l=j)>0$ and $\mathbb{P}_1(Z_m=k)>0$. Thus,
	\beqnn
	\mathbb{P}_k(Z_{n+l+1}=j)\geq\mathbb{P}_k(Z_1=1)\mathbb{P}_1(Z_n=1)\mathbb{P}_1(Z_l=j)
	\eeqnn
	and
	\beqnn
	\mathbb{P}_1(Z_{m+n+1}=1)\geq\mathbb{P}_1(Z_m=k)\mathbb{P}_k(Z_n=j)\mathbb{P}_j(Z_1=1).
	\eeqnn
	Using the logarithm of the expression and letting $n\rightarrow\infty$ yields the first result.
	
	For the second result of the Lemma, first note that $\mathbb{P}_1(Z_n=1)\leq\mathbb{P}_1(1\leq Z_n\leq k_n)$ for $n$ large enough. Thus, it is enough to prove that
	\beqlb\label{knlower}
	\lim_{n\rightarrow\infty}\frac{1}{n}\log\mathbb{P}_1(Z_n=1)\geq\limsup_{n\rightarrow\infty}\frac{1}{n}\log\mathbb{P}_1(1\leq Z_n\leq k_n).
	\eeqlb
	In fact, for $\epsilon>0$, define
	 \beqnn
	\mathcal{A}_\epsilon:=\{\mathbf{q}\in\Delta: H(\{0\})>\epsilon, F(\{0\})>\epsilon, F(\{1\})>\epsilon\}.
	\eeqnn
	According to Assumption (A), $\mathbb{P}(Q\in\mathcal{A}_\epsilon)>0$ if $\epsilon$ is chosen small enough. Thus, we obtain
	\beqnn
	\mathbb{P}_1(Z_n=1)&\geq & \mathbb{P}_1(1\leq Z_{n-1}\leq k_n)\min_{1\leq j\leq k_n}\mathbb{P}_j(Z_1=1)\crcr
	&\geq &\mathbb{P}_1(1\leq Z_{n-1}\leq k_n)\mathbb{P}(Q\in\mathcal{A}_\epsilon) \min_{1\leq j\leq k_n}\mathbb{E}[f(0)^{j-1}f(\{1\})h(0)|Q\in\mathcal{A}_\epsilon] \crcr
	&\geq &\mathbb{P}_1(1\leq Z_{n-1}\leq k_n)\mathbb{P}(Q\in\mathcal{A}_\epsilon)\epsilon^{k_n+1}.
	\eeqnn
	Thus,
	\beqnn
	\lim_{n\rightarrow\infty}\frac{1}{n}\log\mathbb{P}_1(Z_n=1)\geq\limsup_{n\rightarrow\infty}\bigg(\frac{1}{n}\log\mathbb{P}_1(1\leq Z_n\leq k_n)+\frac{k_n+1}{n}\log \epsilon+\frac1{n}\log\mathbb{P}(Q\in\mathcal{A}_\epsilon)\bigg) .
	\eeqnn
	Adding that $k_n=o(n)$ by assumption in Lemma yields (\ref{knlower}), which gives the second result of the Lemma.
\end{proof}

Now, we shall show that the limit in Lemma \ref{smallequal} is negative.

\begin{lemma}\label{<0}
	Assume that  \textbf{Assumptions (A) (B)} hold. Then for all $k,j\geq 1$,
	\beqnn
	 \lim_{n\rightarrow\infty}\frac{1}{n}\log\mathbb{P}_k(Z_n=j):=\varrho\in(-\infty,0).
	\eeqnn
\end{lemma}

\begin{proof}
	From Lemma \ref{smallequal} it is sufficient to prove the Lemma with $k=j=1$.
	 Our proof is based on the decomposition of $\mathbb{P}_1(Z_n=1|\xi)$. We need the following decomposition of $Z_n$:
	\beqnn
	Z_n=Z_n^0+\sum_{i=1}^nZ_{i,n}^0,
	\eeqnn
	where $Z_n^0$ is defined in (\ref{classic}) and for $i\geq 1$, if the environment is $\xi$, $Z_{i,n}^0$ denotes the classic branching process under environment $T^i\xi$ start with $Z_{i,i}^0=Y_i$ and $T$ is the shift translation on the environment: $T(\mathbf{Q}_1,\mathbf{Q}_2,\cdots)=(\mathbf{Q}_2,\mathbf{Q}_3,\cdots)$. More precisely, we have $\mathbb{P}(Z_{i,n}^0\in\cdot|\xi)=\mathbb{P}_{Y_i}(Z_{n-i}^0\in\cdot|T^i\xi)$. Note that given the environment $\xi$, $\{Z_{i,n}^0, 1\leq i\leq n\}$ are independent. Then consider the ancestor of the particle in generation $n$ we have
	\beqnn
	\mathbb{P}_1(Z_n=1|\xi)&=&\mathbb{P}_1(Z_n^0=1;Z_{i,n}^0=0,\forall i|\xi)+\sum_{i=1}^n\mathbb{P}(Z_n^0=0;Z_{i,n}^0=1;Z_{j,n}^0=0, \forall j\neq i|\xi)\crcr
	&\leq &\mathbb{P}_1(Z_n^0=1|\xi)+\sum_{i=1}^{[\frac{n}{2}]}\mathbb{P}(Z_{i,n}^0=1|\xi)+\sum_{i=[\frac{n}{2}]+1}^{n}\mathbb{P}(Z_{j,n}^0=0, \forall j\neq i|\xi)\crcr
	&\leq &\mathbb{P}_1(Z_n^0=1|\xi)+\sum_{i=1}^{[\frac{n}{2}]}\sum_{k=1}^\infty h_i(\{k\})\mathbb{P}_k(Z_{n-i}^0=1|T^i\xi)+\sum_{i=[\frac{n}{2}]+1}^{n}\mathbb{P}(Z_{j,n}^0=0, \forall j\leq \big[\frac{n}{2}\big]|\xi)\crcr
	&\leq & \mathbb{P}_1(Z_n^0=1|\xi)+\sum_{i=1}^{[\frac{n}{2}]}\sum_{k=1}^\infty h_i(\{k\})\mathbb{P}_k(Z_{n-i}^0=1|T^i\xi)+\frac{n}{2}\mathbb{P}(Z_{j,n}^0=0, \forall j\leq \big[\frac{n}{2}\big]|\xi).
	\eeqnn
	Taking expectation on both sides we have
	\beqlb\label{dec11}
	\mathbb{P}_1(Z_n=1)\leq \mathbb{P}_1(Z_n^0=1)+\sum_{i=1}^{[\frac{n}{2}]}\sum_{k=1}^\infty\mathbb{E}[h_i(\{k\})]\mathbb{P}_k(Z_{n-i}^0=1)+\frac{n}{2}\mathbb{E}\bigg[\prod_{i=1}^{[\frac{n}{2}]}h_i(f_{i,n}(0))\bigg].
	\eeqlb
 By \cite[Lemma 4.2, Theorem 2.1 and Propsition 2.2]{2014ban}, there exist constants $C, \beta>0$ such that for all $k\geq 1, n\geq 1$,
	\beqlb\label{firstpart}
	\mathbb{P}_k(Z_n^0=1)\leq Ce^{-\beta n}.
	\eeqlb
	On the other hand, by Lemma \ref{nuj-estimation} we have that there exists $0<\epsilon< \mathbb{E}[\nu(1)]^{-1}$ such that
	\beqnn
	\mathbb{P}(\sharp\{j\geq 0:\nu(j)\leq \big[\frac{n}{2}\big]\}<\epsilon \big[\frac{n}{2}\big])\leq e^{-\alpha [\frac{n}{2}]}.
	\eeqnn
	Thus, use the same method of (\ref{secpart}) we have
	\beqlb\label{secondpart}
	&&\mathbb{E}\bigg[\prod_{i=1}^{[\frac{n}{2}]}h_i(f_{i,n}(0))\bigg]\crcr
	&\leq & \mathbb{P}(\sharp\{j\geq 0:\nu(j)\leq \big[\frac{n}{2}\big]\}<\epsilon \big[\frac{n}{2}\big])+\mathbb{E}\bigg[\prod_{i=1}^{[\frac{n}{2}]}h_i(f_{i,n}(0));\sharp\{j\geq 0:\nu(j)\leq \big[\frac{n}{2}\big]\}\geq\epsilon \big[\frac{n}{2}\big]\bigg]\crcr
	&\leq & e^{-\alpha [\frac{n}{2}]}+(\mathbb{E}[h(1-d)])^{\epsilon [\frac{n}{2}]-1}\crcr
	&\leq & Ce^{-\beta n}
	\eeqlb
	with some constant $C,\beta>0$.
	Combining (\ref{dec11}), (\ref{firstpart}) and (\ref{secondpart}) we have
	\beqnn
	\mathbb{P}_1(Z_n=1)&\leq & Ce^{-\beta n}+C\sum_{i=1}^{[\frac{n}{2}]}e^{-\beta (n-i)}\sum_{k=1}^\infty\mathbb{E}[h_i(\{k\})]\crcr
	&\leq & Ce^{-\beta n}
	\eeqnn
	with some constant $C,\beta>0$, which yields the desired result. 
\end{proof}

  {\it Proof of Theorem \ref{small}.} It is a direct result of Lemmas \ref{smallequal} and \ref{<0}. \qed

  Based on Lemma \ref{<0}, we get the lower deviation of $Z_n$, where we use the  methods in \cite[Lemma 3.1]{2022dar}. 

\begin{lemma}\label{moment}
	Assume that   \textbf{Assumptions (B)}  hold. Then there are $\epsilon_0,\beta_0>0$ and $n_0\in\mathbb{N}$ such that \beqlb\label{st3}
	\mathbb{E}_j[Z_n^{-\epsilon};Z_1\geq n_0,\cdots,Z_n\geq n_0]\leq n_0e^{-\beta_0(n-1)},
	\eeqlb
	for any $0<\epsilon\leq\epsilon_0$, $n\geq 1$ and $j<n_0$.
\end{lemma}

\begin{proof}
	Step 1. We are going to prove that there are $\epsilon_0,\beta_0>0$ and $n_0\in\mathbb{N}$ such that
	\beqlb\label{st2}
	\mathbb{E}_n\bigg[\bigg(\frac{Z_1}{n}\bigg)^{-\epsilon};Z_1\geq 1\bigg]\leq e^{-\beta_0}
	\eeqlb
	for all $n\geq n_0$ and $0<\epsilon\leq\epsilon_0$. Fix $0<\sigma<\min\{1-\delta,1/2\}$. Denote
	 \beqnn
	B_n=\{\mbox{at most }\lfloor\sigma n\rfloor\mbox{ among }N_{1,1},\cdots,N_{1,n}\mbox{ are not equal to }0\}.
	\eeqnn
	From the proof of \cite[Lemma 3.1]{2022dar} we have $\mathbb{P}(B_n|\xi)\leq Ce^{-\beta n}$ for some $\beta>0$ and $C>0$. Thus, we have for every $\epsilon>0$, $n\geq 1$,
%	\beqnn
%	\mathbb{E}_n\bigg[\bigg(\frac{Z_1}{n}\bigg)^{-\epsilon}\mathbf{1}_{\{Z_1\geq 1\}}\mathbf{1}_{B_n}\bigg|\xi\bigg]\leq Cn^{\epsilon}\rho^n, \quad \mathbb{P}-a.s.
%	\eeqnn
%	Therefore, there are $\beta_1>0$ and $C_1>0$ such that
	\beqlb\label{st1}
	\mathbb{E}_n\bigg[\bigg(\frac{Z_1}{n}\bigg)^{-\epsilon}\mathbf{1}_{\{Z_1\geq 1\}}\mathbf{1}_{B_n}\bigg|\xi\bigg]\leq C_1e^{-\beta_1 n}, \quad \mathbb{P}-a.s.
	\eeqlb
%	with $C_1=C\sup_n n^\epsilon\rho^{n/2}, e^{-\beta_1}=\rho^{1/2}$.
Note that
	  \beqnn
	\bigg\{\sum_{i=1}^n N_{1,i}+Y_1\leq \sigma n\bigg\}\subset B_n.
	\eeqnn
	Thus, by (\ref{st1}) we have
	\beqlb\label{<sigman}
	\mathbb{E}_n\bigg[\bigg(\frac{Z_1}{n}\bigg)^{-\epsilon};1\leq Z_1\leq\sigma n\bigg]\leq C_1e^{-\beta_1 n}
	\eeqlb
	for all $n\geq 1$.  By Fatou lemma,
	\beqnn
	\limsup_{n\rightarrow\infty}\mathbb{E}_n\bigg[\bigg(\frac{Z_1}{n}\bigg)^{-\epsilon};Z_1\geq \sigma n\bigg]\leq \mathbb{E}\bigg[\mathbb{E}\bigg(\limsup_{n\rightarrow\infty}\bigg(\frac{1}{n}\sum_{i=1}^n N_{1,i}+\frac{Y_1}{n}\bigg)^{-\epsilon};\sum_{i=1}^n N_{1,i}+Y_1\geq \sigma n\bigg|\xi\bigg)\bigg].
	\eeqnn
	Note that by the strong law of large numbers, given $\xi$,
	\beqnn
	\lim_{n\rightarrow\infty}\bigg(\frac{1}{n}\sum_{i=1}^n N_{1,i}+\frac{Y_1}{n}\bigg)^{-\epsilon}\mathbf{1}_{\{\sum_{i=1}^n N_{1,i}+Y_1\geq \sigma n\}}=m_1^{-\epsilon},\quad \mathbb{P}(\cdot|\xi)-a.s.,
	\eeqnn
 	By Assumption (B), we have $1-\delta\leq m_1,\mathbb{P}$-a.s. and $\mathbb{E}m_1^{-\epsilon}<\infty$. Thus
	\beqlb\label{>sigman}
	\limsup_{n\rightarrow\infty}\mathbb{E}_n\bigg[\bigg(\frac{Z_1}{n}\bigg)^{-\epsilon};Z_1\geq \sigma n\bigg]\leq \mathbb{E}m_1^{-\epsilon}<\infty
	\eeqlb
 for all $\epsilon>0$. Moreover, by the assumption $\mathbb{E}\log m_1>0$,    $$\frac{\mathrm{d}(\mathbb{E}m_1^t)}{\mathrm{d}t}\bigg|_{t=0}=\mathbb{E}\log m_1>0.$$  Then there exists $\epsilon_0$ such that for $0<\epsilon\leq\epsilon_0$,
	\beqnn
	\mathbb{E}m_1^{-\epsilon}<1.
	\eeqnn
  Collecting (\ref{<sigman}) and (\ref{>sigman}), we conclude that there exists $n_0$ such that (\ref{st2}) holds for $n\ge n_0$.
	
	Step 2. Let $\beta_0,\epsilon_0,n_0$ be as in Step 1. We now prove (\ref{st3}) by induction.
	For $n=1$ and $j<n_0$ we have
	\beqnn
	\mathbb{E}_j[Z_1^{-\epsilon};Z_1\geq n_0]&=&\mathbb{E}\bigg[\bigg(\sum_{i=1}^j N_{1,i}+Y_1\bigg)^{-\epsilon}\mathbf{1}_{\{\sum_{i=1}^j N_{1,i}+Y_1\geq n_0\}}\bigg]\crcr
	&\leq & \mathbb{E}\bigg[\bigg(\sum_{i=1}^j N_{1,i}+Y_1\bigg)^{-\epsilon}\cdot\bigg(\sum_{i=1}^j \mathbf{1}_{\{N_{1,i}\geq 1\}}+\mathbf{1}_{\{Y_1\geq 1\}}\bigg)\bigg]\crcr
	&\leq &\sum_{i=1}^j\mathbb{E}[N_{1,i}^{-\epsilon};N_{1,i}\geq 1]+\mathbb{E}[Y_1^{-\epsilon};Y_1\geq 1]\leq n_0.
	\eeqnn
	For arbitrary $n$, using (\ref{st2}) and the induction hypothesis, we obtain
	\beqnn
	&&\mathbb{E}_j[Z_n^{-\epsilon};Z_1\geq n_0,\cdots,Z_n\geq n_0]\crcr
	&&=\mathbb{E}_j\bigg[Z_{n-1}^{-\epsilon}\mathbb{E}_{Z_{n-1}}\bigg[\bigg(\frac{\sum_{i=1}^{Z_{n-1}} N_{n,i}+Y_n}{Z_{n-1}}\bigg)^{-\epsilon};Z_n\geq n_0\bigg];Z_1\geq n_0,\cdots,Z_{n-1}\geq n_0\bigg]\crcr
	&&\leq \mathbb{E}_j[Z_{n-1}^{-\epsilon}e^{-\beta_0};Z_1\geq n_0,\cdots,Z_{n-1}\geq n_0]\crcr
	&&\leq n_0e^{-\beta_0(n-1)},
	\eeqnn
 completing the proof of  (\ref{st3}).
\end{proof}

\textbf{Proof of Theorem \ref{lowerdeviation}:}

%	Step 4. Let $n_0$ be as in Step 2. We will prove that there are $\theta_1<\mathbb{E}\log m_1$ and $\beta_3>0$ such that
%	\beqlb\label{st4}
%	\mathbb{P}_j[Z_n\leq e^{\theta_1 n},Z_1\geq n_0,\cdots,Z_n\geq n_0]\leq Cn_0e^{-\beta_3 n},
%	\eeqlb
%	for any $j<n_0<n$. In fact, 	Then choosing $\theta_1<\min\{\beta_2/\epsilon,\mathbb{E}\log m_1\}$ we conclude (\ref{st4}).
	
%	Step 5. Now we shall prove (\ref{lowerde}).

    Let
	\beqnn
	\tau_n:=\inf\{i\leq n: Z_{i+1}\geq n_0,\cdots,Z_n\geq n_0\},
	\eeqnn
	%which means that $\tau_n$ is the last time smaller than $n$ such that $Z_{\tau_n}<n_0$, and $Z_{\tau_n+1}\geq n_0,\cdots,Z_n\geq n_0$.
We have
	\beqlb\label{fenjie}
	\mathbb{P}_k(1\leq Z_n\leq e^{\theta n})
	&\leq &\sum_{i=1}^{n-1}\mathbb{P}_k(1\leq Z_n\leq e^{\theta n},\tau_n=i)+\mathbb{P}_k(1\leq Z_n\leq n_0)\crcr
	&\leq &\sum_{i=1}^{n-1}\mathbb{P}_k(1\leq Z_n\leq e^{\theta n},\tau_n=i)+C_1e^{-\beta_1 n}
	\eeqlb
	with some $C_1,\beta_1>0$, where we use Theorem \ref{small} in the last inequality. Now we focus on the first part of (\ref{fenjie}). If $i> n/2$, then using Theorem \ref{small}  again and Lemma \ref{pzn=0} we have
	\beqlb\label{>n/2}
	\mathbb{P}_k(1\leq Z_n\leq e^{\theta n},\tau_n=i)\leq \mathbb{P}_k(Z_i=0)+\mathbb{P}_k(1\leq Z_i<n_0)\leq C_1e^{-\beta_1 i}\leq C_2e^{-\beta_2 n}
	\eeqlb 
	with some $C_2,\beta_2>0$.
	
	For $i\leq n/2$, let $\beta_0,\epsilon_0,n_0$ be as in Lemma \ref{moment}. For $0<\epsilon\leq \epsilon_0$,
	choosing $\theta<\min\{\beta_0/\epsilon,\mathbb{E}\log m_1\}$.
	Using (\ref{st3}) and Markov inequality we have for any $j<n_0$,
	\beqlb\label{st4}
	\mathbb{P}_j[Z_n\leq e^{\theta n},Z_1\geq n_0,\cdots,Z_n\geq n_0]&=&\mathbb{P}_j[Z_n^{-\epsilon} \geq e^{-\epsilon \theta n},Z_1\geq n_0,\cdots,Z_n\geq n_0]\crcr
	&\leq & e^{\epsilon\theta n}\mathbb{E}_j[Z_n^{-\epsilon};Z_1\geq n_0,\cdots,Z_n\geq n_0]\crcr
	&\leq & e^{\epsilon\theta n}n_0e^{-\beta_0(n-1)}=n_0e^{\epsilon\theta}e^{(\epsilon\theta-\beta_0)(n-1)}\crcr
	&\leq &C_3n_0e^{-\beta_3 n}
	\eeqlb
	for some $C_3,\beta_3>0$.
	By Markov property we obtain
	  \beqlb\label{<n/2}
	\mathbb{P}_k(1\leq Z_n\leq e^{\theta n},\tau_n=i)
	&\leq &\mathbb{P}_k(0\leq Z_i<n_0)\sup_{j< n_0}\mathbb{P}_j(1\leq Z_{n-i}\leq e^{\theta n},Z_1\geq n_0,\cdots,Z_{n-i}\geq n_0)\crcr
	&\leq &C_3n_0e^{-\beta_3 n/2},
	\eeqlb 
	where we use (\ref{st4}) in the last inequality.  Collecting (\ref{fenjie}), (\ref{>n/2}) and (\ref{<n/2}), we obtain the result.
\qed

\section{Limit theorems for $\log Z_n$}

\subsection{Main results}
Now we focus on the limit theorems of $\log Z_n$. We need the following assumption in this section:

  \noindent\textbf{Assumption (C)} \quad There are $q>1$, $p\in(1,2]$ such that
\beqnn
\mathbb{E}\bigg[(1+|\log m_{1}|^q)\bigg(\bigg(\frac{N_{1,1}}{m_{1}}\bigg)^p+1\bigg)\bigg]<\infty \quad \mbox{and}\quad \mathbb{E}[Y_1^{p}]<\infty.
\eeqnn

Let $\lambda(\cdot)$ be the characteristic function of   $\log m_1$, i.e.,
\beqlb\label{lambda}
\lambda(s):=\mathbb{E}[e^{\mathrm{i}s\log m_1}]=\mathbb{E}[m_1^{\mathrm{i}s}].
\eeqlb
Denote $\mu:=\mathbb{E}\log m_1$ and $\sigma:=Var(\log m_1)=-\lambda''(0)+\lambda'(0)^2.$ 

 Observing that from Lemma \ref{pzn=0}, $\mathbb{P}_k(Z_n>0)\rightarrow 1$ as $n\rightarrow\infty$. We give the central limit theorem of $\log Z_n$ as follows:

%Let $\mathbb{E}_\mathcal{S}[\cdot]=\mathbb{E}[\cdot|\mathcal{S}]$ denotes the expected value conditioned on the survival set $\mathcal{S}$.

\begin{theorem}\label{CLT}
 	  Assume that $\mathbb{E}(\log m_1)^2<\infty$, $\sigma>0$  and \textbf{Assumptions (A)--(C)} hold. Then for all $k\geq 1$,
 	\beqlb
 	\mathbb{P}_k\bigg(\frac{\log Z_n-n\mu}{\sqrt{n}\sigma}\leq x\bigg|Z_n>0\bigg)\rightarrow \Phi(x), \quad \forall x\in\mathbb{R},
 	\eeqlb
 	where
 	\beqnn
 	\Phi(x):=\frac{1}{\sqrt{2\pi}}\int_{-\infty}^x e^{-y^2/2}dy,\quad x\in\mathbb{R}.
 	\eeqnn
 \end{theorem}

 For classical BPRE $\mathbf{Z^0}$, Buraczewski and Damek \cite{2022dar} proved the central limit theorem of $\log Z_n^0$ conditionally on the survival set,  with the hypothesis $\mathbb{P}(Z_1^0=0)=0$. 

 Next result gives  the Edgeworth expansion of the BPIRE $\mathbf{Z}$:
 \begin{theorem}\label{edgeworth}

 	Suppose that $\log m_1$ is nonlattice,   \textbf{Assumptions (A) (B)} hold and \textbf{Assumption (C)} is satisfied with $q\geq 4$. 
  Let $r $ be a positive integer. If   $r=3$, or $4\le r\leq q-1$ and
 	  \beqlb\label{wuqiong}
 	\limsup_{|s|\rightarrow\infty}|\mathbb{E}m_1^{\mathrm{i}s}|<1,
 	\eeqlb 
 	then
  \beqnn
 	\mathbb{P}\bigg(\frac{\log Z_n-n\mu}{\sqrt{n}\sigma}\leq x\bigg|Z_n>0\bigg)=G_r(x)+o(n^{-r/2+1}),\qquad n\to \infty.
 	\eeqnn 
 	where
 	\beqnn
 	G_r(x)=\Phi(x)-\varphi(x)\sum_{k=3}^r n^{-k/2+1}P_k(x),
 	\eeqnn
  $\varphi(x):=\frac{1}{\sqrt{2\pi}}e^{-x^2/2} $, $\Phi$ is defined as in Theorem \ref{CLT}, $P_k$ is a polynomial of order $k-1$ independent of $n$ and $r$ and $o(n^{-r/2+1})$ denotes a function of order smaller than $n^{-r/2+1}$ uniformly with respect to $x$.
 \end{theorem}

  The following is a renewal theorem for $\log Z_n$:
 \begin{theorem}\label{renewal}
 Assume that   $\log m_1$ is nonlattice and \textbf{Assumptions (A)--(C)} hold. Then for all  $k\ge 1$ and  $0\leq B<C<\infty$,
 \beqnn
 \lim_{y\rightarrow\infty}\mathbb{E}_k\bigg(\sharp\{n:\log Z_n\in y+[B,C]\}\bigg)= \lim_{y\rightarrow\infty}\mathbb{E}_k\bigg(\sharp\{n:e^B y\leq Z_n\leq e^C y\}\bigg)=\frac{1}{\mu}(C-B).
 \eeqnn
 	
 \end{theorem}

 \subsection{Auxiliary results}
 Using Theorem \ref{lowerdeviation}, we estimate the harmonic moment of $\mathbf{Z}$ and the moment of $\log \mathbf{Z}$:

\begin{lemma}\label{harmonicZ}(Harmonic moment of $\textbf{Z}$)
	Suppose that \textbf{Assumptions (A) (B)} hold. Then for any $\alpha>0$ there are constants $C(\alpha),\beta>0$ such that for all $k\geq 1$,
	\beqlb\label{harmon}
	\mathbb{E}_k[Z_n^{-\alpha};Z_n>0]\leq C(\alpha)e^{-\beta n}.
	\eeqlb
\end{lemma}

\begin{proof}
	By Theorem \ref{lowerdeviation}, for all $k\geq 1$,
	\beqnn
	\mathbb{E}_k[Z_n^{-\alpha};Z_n>0]
	&\leq &\mathbb{P}_k(1\leq Z_n\leq e^{\theta n})+\mathbb{E}_k[Z_n^{-\alpha};Z_n\geq e^{\theta n}]\crcr
	&\leq &Ce^{-\beta n}+e^{-\alpha\theta n}
	\eeqnn
	with $\theta,\beta$ as in (\ref{lowerde}). Then we get the desired result.
\end{proof}

\begin{lemma}(Moment of $\log\mathbf{Z}$) \label{logzn}
	Assume that    \textbf{Assumptions (A) (B)} hold.  Then for any $j\geq 1$ there are constants  $C(j)>0$ and $\beta>0$ such that for all $k\geq 1$,
	\beqnn
	\mathbb{E}_k[(\log Z_n)^j;Z_{n+1}=0,Z_n>0]\leq  C(j)e^{-\beta n}.
	\eeqnn
\end{lemma}

\begin{proof}
	By branching property and Assumption (B) we have
	\beqnn
	\mathbb{E}_k[(\log Z_n)^j;Z_{n+1}=0,Z_n>0]&=&\mathbb{E}_k[(\log Z_n)^j\mathbf{1}_{\{Z_n>0\}}\mathbb{E}[\mathbf{1}_{\{Z_{n+1}=0\}}|Z_n,\xi]]\crcr
	&=&\mathbb{E}_k[(\log Z_n)^j f_{n+1}^{Z_n}(0)h_{n+1}(0);Z_n>0]\crcr
	&\leq &\mathbb{E}_k[(\log Z_n)^j \delta^{Z_n};Z_n>0].
	\eeqnn
	Since $\delta\in(0,1)$, we have $\frac{\delta+1}{2\delta}>1$. Thus, for any $j\geq 1$ there is a constant $\bar{C}(j)$ such that
	\beqnn
	(\log x)^j\leq \bar{C}(j)\bigg(\frac{\delta+1}{2\delta}\bigg)^x,\quad x\geq 1.
	\eeqnn
	Using this inequality we have
	\beqnn
	\mathbb{E}_k[(\log Z_n)^j;Z_{n+1}=0,Z_n>0]
	&\leq &
	\bar{C}(j)\mathbb{E}_k\bigg[\bigg(\frac{\delta+1}{2}\bigg)^{Z_n};Z_n>0\bigg]\crcr
	&\leq &\bar{C}(j)\mathbb{E}_k\bigg[\bigg(\frac{\delta+1}{2}\bigg)^{Z_n};Z_n>e^{\theta n}\bigg]+\bar{C}(j)\mathbb{P}_k(1\leq Z_n\leq e^{\theta n}) \crcr
	&\leq &\bar{C}(j)\bigg(\frac{\delta+1}{2}\bigg)^{e^{\theta n}}+\bar{C}(j)Ce^{-\beta n}\crcr
	\eeqnn
	with $\theta,\beta$ as in (\ref{lowerde}), where we use Theorem \ref{lowerdeviation} in the last inequality. Combining this with the fact $(\delta+1)/2\in(0,1)$ yields the desired result.
\end{proof}

Next we use the skill in \cite{2022dar} to measure deviation of the process $Z_{n+1}$ from $m_{n+1}Z_n$ on the set $\{Z_n>0\}$. Let us define
\beqlb\label{Delta}
\Delta_n=\frac{Z_{n+1}}{m_{n+1}Z_n}.
\eeqlb
Then $\Delta_n$ makes sense only when $Z_n>0$. %Note that on the set $\{Z_n>0\}$, $\mathbb{E}[Z_{n+1}|\xi,Z_n]=m_{n+1}Z_n+\lambda_{n+1}$.
We hope that $\Delta_n$ should be close to $1$.

\begin{lemma}\label{deviation}
	Suppose  that  \textbf{Assumptions (A)--(C)} hold. Then there are constants $C,\beta>0$ such that for any $r\in[0,q]$, $k\geq 1$,
	\beqnn
	\mathbb{E}_k[|\log(m_{n+1}Z_n)|^r|\Delta_n-1|^p;Z_n>0]\leq Ce^{-\beta n}.
	\eeqnn
%	with $p$ in {\red Assumption (C)}.
\end{lemma}

\begin{proof}
	From (\ref{Delta}) and (\ref{define}), on the set $\{Z_n>0\}$, we have the decomposition
	\beqnn
	\Delta_n=\frac{Z_{n+1}}{m_{n+1}Z_n}=\frac{1}{Z_n}\bigg(\sum_{i=1}^{Z_n}\frac{N_{n+1,i}}{m_{n+1}}+\frac{Y_{n+1}}{m_{n+1}}\bigg)
	.
	\eeqnn
	Then we use a direct consequence of the Marcinkiewicz-Zygmund inequality (see \cite[Lemma 2.1]{2017liu}) to estimate the conditional expectation of $|\Delta_n-1|$ on the set $\{Z_n>0\}$   in quenched law:
	\beqnn
	&&\mathbb{E}_k[|\log m_{n+1}|^r|\Delta_n-1|^p|Z_n,\xi]\crcr
	&&\leq C\mathbb{E}_k\bigg[\frac{|\log m_{n+1}|^r}{Z_n^p}\bigg|\sum_{i=1}^{Z_n}\bigg(\frac{N_{n+1,i}}{m_{n+1}}-1\bigg)\bigg|^p\big|Z_n,\xi\bigg]+CZ_n^{-p}\mathbb{E}\bigg[|\log m_{n+1}|^r\bigg|\frac{Y_{n+1}}{m_{n+1}}\bigg|^p|\xi\bigg] \crcr
	&&\leq  CZ_n^{1-p}\mathbb{E}\bigg[|\log m_{n+1}|^r\bigg|\frac{N_{n+1,i}}{m_{n+1}}-1\bigg|^p|\xi\bigg]+CZ_n^{-p}\mathbb{E}\bigg[|\log m_{n+1}|^r\bigg|\frac{Y_{n+1}}{m_{n+1}}\bigg|^p|\xi\bigg] .
	\eeqnn
	Therefore, taking expectation we obtain
	\beqnn
	&&\mathbb{E}_k[|\log(m_{n+1}Z_n)|^r|\Delta_n-1|^p;Z_n>0]\crcr
	&&\leq C\mathbb{E}_k\bigg[\bigg(|\log m_{n+1}|^r+(\log Z_n)^r\bigg)|\Delta_n-1|^p;Z_n>0\bigg]\crcr
	&&\leq C\mathbb{E}_k[Z_n^{1-p}(1+(\log Z_n)^r);Z_n>0]\cdot\mathbb{E}\bigg[(1+|\log m_{1}|^r)\bigg(\bigg|\frac{N_{1,1}}{m_{1}}-1\bigg|^p+\bigg|\frac{Y_1}{m_1}\bigg|^p\bigg)\bigg].
	\eeqnn
	Combining this with Lemma \ref{harmonicZ} and Assumption (C) gives the claim in the lemma. 
\end{proof}

Next Lemma is the direct consequence of Lemma \ref{deviation}:

\begin{lemma}\label{logdeltan}
Suppose  that   \textbf{Assumptions (A)--(C)} hold. Then there are constants $C,\beta >0$ such that for all $k\geq 1$,
	\beqnn
	\mathbb{E}_k[(\log Z_{n+1})^j|\log (m_{n+1}Z_n)|^l|\log \Delta_n|;Z_{n+1}>0,Z_n>0]\leq Ce^{-\beta n}
	\eeqnn
	for any $n,j,l\in\mathbb{N}$ and $j+l+1\leq q$.
\end{lemma}

\begin{proof}
    Note that $Z_{n+1}=\Delta_n m_{n+1}Z_n$. On the set $\{Z_{n+1}>0,Z_n>0\}$ we have
    \beqnn
    \log Z_{n+1}\leq |\log\Delta_n|\mathbf{1}_{\{\Delta_n\geq 1/2\}}+|\log(m_{n+1}Z_n)|.
    \eeqnn
    Thus, it is sufficient to prove
    \beqlb\label{del>1/2}
    \mathbb{E}_k[|\log (m_{n+1}Z_n)|^l|\log \Delta_n|^{j+1};\Delta_n\geq 1/2,Z_{n+1}>0,Z_n>0]\leq Ce^{-\beta n}
    \eeqlb
    for $j+l+1\leq q$ 
    and
    \beqlb\label{del<1/2}
    \mathbb{E}_k[|\log (m_{n+1}Z_n)|^l|\log \Delta_n|;\Delta_n<1/2,Z_{n+1}>0,Z_n>0]\leq Ce^{-\beta n}
    \eeqlb
    for $l+1\leq q$.

    Firstly we focus on (\ref{del>1/2}).  For  $s>1$ and $t\le s$, noting that $$
    J(x):=\frac{|\log(1+x)|^{s}}{|x|^t}
    $$
    is  continuous on $[-1/2,\infty)$  and $\lim_{x\to\infty} J(x)=0$, then
    there is a constant $0<C<\infty$ such that
    $J(x)\leq C$
	for any $x\geq -1/2$.  Thus, if $j>0$, taking $s=j+1$ and $t=p$ in $J(x)$ we obtain
	\beqnn
	&&\mathbb{E}_k[|\log (m_{n+1}Z_n)|^l|\log \Delta_n|^{j+1};\Delta_n\geq 1/2,Z_{n+1}>0,Z_n>0]\crcr
	&\leq & C\mathbb{E}_k[|\log (m_{n+1}Z_n)|^l| \Delta_n-1|^p;Z_n>0]\crcr
	&\leq & Ce^{-\beta n},
	\eeqnn
	where we use Lemma \ref{deviation} in the last inequality. If $j=0$, then choose $1<r<p$ such that $rl\leq q$. Using the Jensen inequality we obtain
	\beqnn
	&&\mathbb{E}_k[|\log (m_{n+1}Z_n)|^l|\log \Delta_n|;\Delta_n\geq 1/2,Z_{n+1}>0,Z_n>0]\crcr
	&\leq &(\mathbb{E}_k[|\log (m_{n+1}Z_n)|^{rl}|\log \Delta_n|^r;\Delta_n\geq 1/2,Z_{n+1}>0,Z_n>0])^{1/r}.
	\eeqnn
	Then  taking $s=r$ and $t=r$ in $J(x)$,  using the same methods as in the previous case, we finish the proof of (\ref{del>1/2}).

    To estimate (\ref{del<1/2}) note that on the set $\{Z_{n+1}>0,Z_n>0\}\cap\{\Delta_n<1/2\}$ we have
	\beqnn
	\frac{1}{2}>\Delta_n=\frac{Z_{n+1}}{m_{n+1}Z_n}\geq\frac{1}{m_{n+1}Z_n},
	\eeqnn
 which implies that
 $$|\log(m_{n+1}Z_n)|\geq |\log\Delta_n|,\quad 2^p|\Delta_n-1|^p\geq 1.$$
 Thus, from Lemma \ref{deviation},
	\beqnn
	&&\mathbb{E}_k[|\log (m_{n+1}Z_n)|^l|\log \Delta_n|;\Delta_n<1/2,Z_{n+1}>0,Z_n>0]\crcr
	&\leq & 2^p\mathbb{E}_k[|\log(m_{n+1}Z_n)|^{l+1}|\Delta_n-1|^p;Z_n>0]\crcr
	&\leq &  Ce^{-\beta n}.
	\eeqnn
	Consequently, the proof is finished.
\end{proof}

 In order to prove Theorems \ref{edgeworth} and \ref{renewal}, we shall use the upper bound of the characteristic function of $\log Z_n$ on $\{Z_n>0\}$ below. The proof is similar with the proof of Lemma 4.1 in \cite{2022dar}. We omit the details here. 
%we use the connection between $Z_n$ and $S_n$.

\begin{lemma}
	  Suppose  that  \textbf{Assumptions (A)--(C)} hold, $\mathbb{E}(\log m_1)^2<\infty$, $\sigma>0$ and $\log m_1$ is nonlattice. Then, given $0<M<\infty$, there are positive constants $C_1,C_2$ such that for all $k\geq 1$,
	\beqlb\label{n-13M}
	\sup_{|s|\in[n^{-1/3},M]}\left|\mathbb{E}_k[Z_n^{\mathrm{i}s};Z_n>0]\right|\leq C_1e^{-C_2 n^{1/12}}.
	\eeqlb
	Moreover if (\ref{wuqiong}) holds, then for any $\kappa_1,\kappa_2>0$ there are constants $C_3,\beta>0$ such that for all $k\geq 1$,
	\beqlb\label{kngamma}
	\sup_{|s|\in[\kappa_1,n^{\kappa_2}]}\left|\mathbb{E}_k[Z_n^{\mathrm{i}s};Z_n>0]\right|\leq C_3e^{-\beta n}.
	\eeqlb
\end{lemma}

\subsection{Proofs}

Recall $\lambda(\cdot)$ defined in (\ref{lambda}). Denote
\beqnn
\phi_{k,n}(s):=\frac{\mathbb{E}_k[e^{\mathrm{i}s\log Z_n};Z_n>0]}{\lambda(s)^n}=\frac{\mathbb{E}_k[Z_n^{\mathrm{i}s};Z_n>0]}{\lambda(s)^n},\quad k\geq 1,s\in\mathbb{R}.
\eeqnn
Now we present a significant lemma which is key to the proof of Theorems \ref{CLT}, \ref{edgeworth} and \ref{renewal}:

\begin{lemma}\label{key}
	Assume that  \textbf{Assumptions (A)--(C)} hold. Let $K=\lfloor q-1\rfloor$. Then there are $\eta>0$, a function $\phi\in C^K(I_\eta)$ defined on $I_\eta:=(-\eta,\eta)$ and  constants $C,\beta>0$  such that for all $k\geq 1$,
	\beqlb\label{keyineq}
	|\phi_{k,n}^{(j)}(s)-\phi^{(j)}(s)|\leq   Ce^{-\beta n}, \quad s\in I_\eta,j=0,\cdots,K.
	\eeqlb
\end{lemma}

\begin{proof}
  The frame work of our proof is inspired by \cite[Proposition 2.1]{2022dar}.
	
	Step 1. We show that
	%It is sufficient to prove that
	\beqlb\label{suff}
	|\phi_{k,n+1}(s)-\phi_{k,n}(s)|\leq Ce^{-\beta_0 n},\quad s\in I_{\eta_0}
	\eeqlb
	for some $C,\beta_0>0$ and $\eta_0>0$. This entails the existence of $\phi$. Using the independence of $Z_n$ and $m_{n+1}$ we have
	\beqlb\label{phi-phi}
	&&|\lambda(s)|^{n+1}|\phi_{k,n+1}(s)-\phi_{k,n}(s)|\crcr
	&=&|\mathbb{E}_k[Z_{n+1}^{\mathrm{i}s};Z_{n+1}>0]-\lambda(s)\mathbb{E}_k[Z_n^{\mathrm{i}s};Z_n>0]|\crcr
	&\leq &\mathbb{E}_k[|Z_{n+1}^{\mathrm{i}s}-(m_{n+1}Z_n)^{\mathrm{i}s}|;Z_{n+1}>0,Z_n>0]\crcr
	&&+\mathbb{E}_k[|Z_{n+1}^{\mathrm{i}s}|;Z_{n+1}>0,Z_n=0]+\mathbb{E}_k[|(m_{n+1}Z_n)^{\mathrm{i}s}|;Z_{n+1}=0,Z_n>0]\crcr
	&\leq &\mathbb{E}_k[|Z_{n+1}^{\mathrm{i}s}-(m_{n+1}Z_n)^{\mathrm{i}s}|;Z_{n+1}>0,Z_n>0]\crcr
	&&+\mathbb{P}_k(Z_n=0)+\mathbb{P}_k(Z_{n+1}=0).
	\eeqlb
	From Lemma \ref{pzn=0} we have
	\beqlb\label{pn+1}
	\mathbb{P}_k(Z_n=0)+\mathbb{P}_k(Z_{n+1}=0)\leq Ce^{-\beta n}
	\eeqlb
	with some $C,\beta>0$.
	  Note that \beqlb\label{basicineq}
	|e^{\mathrm{i}s}-e^{\mathrm{i}t}|\leq |s-t|,\quad s,t\in\mathbb{R}.
	\eeqlb 
	From Lemmas \ref{logdeltan} and (\ref{basicineq}) we obtain
	\beqnn
	\mathbb{E}_k[|Z_{n+1}^{\mathrm{i}s}-(m_{n+1}Z_n)^{\mathrm{i}s}|;Z_{n+1}>0,Z_n>0]&\leq &|s|\mathbb{E}[|\log\Delta_n|;Z_{n+1}>0,Z_n>0]\crcr
	&\leq &C|s|e^{-\beta n}
	\eeqnn
	with some $C,\beta>0$.
	Combining this with (\ref{phi-phi}) and (\ref{pn+1}) yields that
	\beqnn
	|\lambda(s)|^{n+1}|\phi_{k,n+1}(s)-\phi_{k,n}(s)|\leq Ce^{-\beta n}
	\eeqnn
	with some $C,\beta>0$ for $s\in I_\eta$. Since the function $\lambda(\cdot)$ is continuous and $\lambda(0)=1$, there exists a small neighborhood of $0$ such that (\ref{suff}) is satisfied with some $\beta_0>0$.

	Step 2. We are going to prove that there are constants $C>0,\eta_1>0,\beta_1>0$ such that
	\beqlb\label{j=1}
	|\lambda(s)|^{n+1}|\phi_{k,n+1}'(s)-\phi_{k,n}'(s)|\leq Ce^{-\beta_1 n},\quad n\in\mathbb{N}, s\in I_{\eta_1}.
	\eeqlb
	In fact, from the definition of $\phi_{k,n}$, we have
	\beqlb\label{bds}
	\lambda(s)^{n+1}(\phi_{k,n+1}(s)-\phi_{k,n}(s))=\mathbb{E}_k[Z_{n+1}^{\mathrm{i}s};Z_{n+1}>0]-\lambda(s)\mathbb{E}_k[Z_n^{\mathrm{i}s};Z_n>0].
	\eeqlb
	Denote
	\beqnn
	L_{k,n}(s):=\lambda(s)^{n+1}(\phi_{k,n+1}(s)-\phi_{k,n}(s))
	\eeqnn
	and
	\beqnn
	J_{k,n}(s):=\mathbb{E}_k[Z_{n+1}^{\mathrm{i}s};Z_{n+1}>0]-\lambda(s)\mathbb{E}_k[Z_n^{\mathrm{i}s};Z_n>0].
	\eeqnn
	Differentiating the function $L_{k,n}(s)$ we obtain
	\beqnn
	L_{k,n}'(s)=(n+1)\lambda'(s)\lambda(s)^n(\phi_{k,n+1}(s)-\phi_{k,n}(s))+\lambda(s)^{n+1}(\phi_{k,n+1}'(s)-\phi_{k,n}'(s)).
	\eeqnn
	Thus, it is sufficient to prove that there are constants $C,\eta,\beta>0$ such that
	\beqlb\label{suff1}
	|(n+1)\lambda'(s)\lambda(s)^n(\phi_{k,n+1}(s)-\phi_{k,n}(s))|\leq Ce^{-\beta n}, \quad
	s\in I_{\eta},
	\eeqlb
	and
	\beqlb\label{suff2}
	|J_{k,n}'(s)|\leq Ce^{-\beta n},\quad
	s\in I_{\eta}.
	\eeqlb
	Since $|\lambda'(s)|\leq \mathbb{E}| \log m_1 |<\infty$, in view of (\ref{suff}),
	\beqnn
	|(n+1)\lambda'(s)\lambda(s)^n(\phi_{k,n+1}(s)-\phi_{k,n}(s))|&\leq & C(n+1)e^{-\beta_0 n}|\lambda(s)|^n\mathbb{E}|  \log m_1 |, \quad \text{for} \quad s\in I_{\eta_0}\crcr
	&\leq &Ce^{-\beta n}, \quad \text{for} \quad s\in I_{\eta}
	\eeqnn
	where $\eta$ is taken small enough to ensure $e^{-\beta_0}|\lambda(s)|\leq e^{-\beta}<1$ for $s\in I_{\eta}$.
	
	For (\ref{suff2}), noticing that by the independence of $Z_n$ and $m_{n+1}$,
	\beqnn
	J_{k,n}'(s)&=&\mathbb{E}_k[\mathrm{i}\log Z_{n+1}\cdot Z_{n+1}^{\mathrm{i}s};Z_{n+1}>0]-\lambda'(s)\mathbb{E}_k[Z_n^{\mathrm{i}s};Z_n>0]-\lambda(s)\mathbb{E}_k[\mathrm{i}\log Z_n\cdot Z_n^{\mathrm{i}s};Z_n>0]\crcr
	&=&\mathbb{E}_k[\mathrm{i}\log Z_{n+1}\cdot Z_{n+1}^{\mathrm{i}s};Z_{n+1}>0]-\mathbb{E}_k[\mathrm{i}\log m_{n+1}\cdot(m_{n+1}Z_n)^{\mathrm{i}s};Z_n>0]\crcr
	&&-\mathbb{E}_k[\mathrm{i}\log Z_n\cdot(m_{n+1}Z_n)^{\mathrm{i}s};Z_n>0]\crcr
	&=&\mathbb{E}_k[\mathrm{i}\log Z_{n+1}\cdot Z_{n+1}^{\mathrm{i}s};Z_{n+1}>0]-\mathbb{E}_k[\mathrm{i}\log (m_{n+1}Z_n)\cdot(m_{n+1}Z_n)^{\mathrm{i}s};Z_n>0].
	\eeqnn
	Thus, using (\ref{basicineq}) and the independence of $(Y_{n+1},m_{n+1})$ and $Z_n$ we obtain
	\beqnn
	&&|J_{k,n}'(s)|\crcr
	&\leq & |\mathbb{E}_k[\mathrm{i}\log Z_{n+1}\cdot Z_{n+1}^{\mathrm{i}s};Z_{n+1}>0,Z_n>0]-\mathbb{E}_k[\mathrm{i}\log (m_{n+1}Z_n)\cdot Z_{n+1}^{\mathrm{i}s};Z_{n+1}>0,Z_n>0]|\crcr
	&&+|\mathbb{E}_k[\mathrm{i}\log (m_{n+1}Z_n)\cdot Z_{n+1}^{\mathrm{i}s};Z_{n+1}>0,Z_n>0]\crcr
	&&   -\mathbb{E}_k[\mathrm{i}\log (m_{n+1}Z_n)\cdot(m_{n+1}Z_n)^{\mathrm{i}s}; Z_{n+1}>0,Z_n>0]|\crcr
	&&+|\mathbb{E}_k[\mathrm{i}\log Z_{n+1}\cdot Z_{n+1}^{\mathrm{i}s};Z_{n+1}>0,Z_n=0]|+|\mathbb{E}_k[\mathrm{i}\log (m_{n+1}Z_n)\cdot (m_{n+1}  Z_{n} )^{\mathrm{i}s};Z_{n+1}=0,Z_n>0]|\crcr
	&\leq & \mathbb{E}_k[|\log\Delta_n|;Z_{n+1}>0,Z_n>0]+|s|\mathbb{E}_k[|\log(m_{n+1}Z_n)||\log\Delta_n|;Z_{n+1}>0,Z_n>0]\crcr
	&&+\mathbb{E}[\log Y_1;Y_1>0]\mathbb{P}_k(Z_n=0)+(\mathbb{E} |\log m_1|^q)^{1/q}(\mathbb{P}_k(Z_{n+1}=0))^{1/r}\crcr
	&&+\mathbb{E}_k[\log Z_n;Z_{n+1}=0,Z_n>0],
	\eeqnn
	where we use H$\ddot{o}$lder's inequality in the last inequality with $q>1$ taken as in  Assumption (C) and $r>1$ such that $1/r+1/q=1$.
	Note that   $\mathbb{E}\log m_1<\infty$ and by  Assumption (C) $\mathbb{E}[\log Y_1;Y_1>0]<\infty$. Applying Lemmas \ref{pzn=0}, \ref{logzn} and \ref{logdeltan} we have
	\beqnn
	|J_{k,n}'(s)|\leq Ce^{-\beta n},\quad s\in I_\eta
	\eeqnn
	for some $\beta,C>0$. Therefore,  we conclude (\ref{suff1}) and (\ref{suff2}), which yield (\ref{j=1}).
	
	Step 3. We shall prove the Lemma by induction. We are going to prove that for any $j\leq K$ there are $C,\eta_j,\beta_j>0$ such that for all $k\geq 1$,
	\beqlb\label{keykey}
	|\lambda(s)|^{n+1}|\phi_{k,n+1}^{(j)}(s)-\phi_{k,n}^{(j)}(s)|\leq Cn^je^{-\beta_j n}, \quad n\in\mathbb{N},s\in I_{\eta_j}.
	\eeqlb
	We finish the proof of the case when $j=1$ in step 2, which implies that the sequence of derivatives $\phi_{k,n}'$ converges uniformly to some function $\psi$ on $I_{\eta_1}$. Therefore $\phi'=\psi$ and $\phi$ is continuously differentiable (see e.g. \cite[Theorem 14.7.1]{Tao2006}). The same inductive argument guarantees that $\phi\in C^K(I_{\eta_K})$ and since $\lambda$ is continuous with $\lambda(0)=1$, inequality (\ref{keyineq}) follows from (\ref{keykey}) with $\min_{j\leq K} \beta_j\geq \beta>0$ and $\eta<\min_{j\leq K}\eta_j$.
	
	Suppose that (\ref{keykey}) holds for $j\leq l-1$ and recall (\ref{bds}). Using the binomial formula for $L_{k,n}^{(l)}(s)$, to prove (\ref{keykey}) it is sufficient to prove
	\beqlb\label{key1}
	|L_{k,n}^{(l)}(s)-\lambda(s)^{n+1}(\phi_{k,n+1}^{(l)}(s)-\phi_{k,n}^{(l)}(s))|&=&\bigg|\sum_{j=0}^{l-1}\tbinom{l}{j}(\lambda(s)^{n+1})^{(l-j)}(\phi_{k,n+1}^{(j)}(s)-\phi_{k,n}^{(j)}(s))\bigg|\crcr
	&\leq & Cn^le^{-\beta n}
	\eeqlb
	and
	\beqlb\label{key2}
	|J_{k,n}^{(l)}(s)|\leq C^{-\beta n}
	\eeqlb
	for some $C,\beta>0$. In fact, by the induction hypothesis and (3.21) in \cite{2022dar} we have for $j\leq l-1$,
	\beqnn
	|(\lambda(s)^{n+1})^{(l-j)}(\phi_{k,n+1}^{(j)}(s)-\phi_{k,n}^{(j)}(s))|=\bigg|\frac{(\lambda(s)^{n+1})^{(l-j)}}{\lambda(s)^{n+1}}\bigg||\lambda(s)^{n+1}||\phi_{k,n+1}^{(j)}(s)-\phi_{k,n}^{(j)}(s)|\crcr
	\leq C_{l-j}n^{l-j}Cn^je^{-\beta n},
	\eeqnn
	which yields (\ref{key1}).
	
	For (\ref{key2}), note that
	\beqnn
	J_{k,n}^{(l)}(s)=\mathbb{E}_k[(\mathrm{i}\log Z_{n+1})^l\cdot Z_{n+1}^{\mathrm{i}s};Z_{n+1}>0]-\mathbb{E}_k[(\mathrm{i}\log (m_{n+1}Z_n))^l\cdot(m_{n+1}Z_n)^{\mathrm{i}s};Z_n>0].
	\eeqnn
	Thus, using (\ref{basicineq}), H\"{o}lder's inequality with $1/r+l/q=1$ and the independence of $(Y_{n+1},m_{n+1})$ and $Z_n$ we obtain
	\beqlb\label{jnk}
	|J_{k,n}^{(l)}(s)|&\leq &|\mathbb{E}_k[(\mathrm{i}\log Z_{n+1})^l\cdot Z_{n+1}^{\mathrm{i}s};Z_{n+1}>0,Z_n>0]-\mathbb{E}_k[(\mathrm{i}\log (m_{n+1}Z_n))^l\cdot Z_{n+1}^{\mathrm{i}s};\crcr
	&& \quad \quad Z_{n+1}>0,Z_n>0]|\crcr
	&&+|\mathbb{E}_k[(\mathrm{i}\log (m_{n+1}Z_n))^l\cdot Z_{n+1}^{\mathrm{i}s};Z_{n+1}>0,Z_n>0]-\mathbb{E}_k[(\mathrm{i}\log (m_{n+1}Z_n))^l\cdot(m_{n+1}Z_n)^{\mathrm{i}s};\crcr
	&& \quad \quad Z_{n+1}>0,Z_n>0]|\crcr
	&&+|\mathbb{E}_k[(\mathrm{i}\log Z_{n+1})^l\cdot Z_{n+1}^{\mathrm{i}s};Z_{n+1}>0,Z_n=0]|\crcr
	&&+|\mathbb{E}_k[(\mathrm{i}\log (m_{n+1}Z_n))^l\cdot (m_{n+1}Z_n)^{\mathrm{i}s};Z_{n+1}=0,Z_n>0]|\crcr
	&\leq &\mathbb{E}_k[|(\log Z_{n+1})^l-(\log (m_{n+1}Z_n))^l|;Z_{n+1}>0,Z_n>0]\crcr
	&&+|s|\mathbb{E}_k[|\log(m_{n+1}Z_n)|^l|\log\Delta_n|;Z_{n+1}>0,Z_n>0]\crcr
	&&+\mathbb{E}_k[(\log Y_1)^l;Y_1>0]\mathbb{P}_k(Z_n=0)+C(\mathbb{E}[| \log m_1|^q])^{l/q}(\mathbb{P}_k(Z_{n+1}=0))^{1/r}\crcr
	&&+C\mathbb{E}_k[(\log Z_n)^l;Z_{n+1}=0,Z_n>0]  \crcr
	&\leq & \mathbb{E}_k[|(\log Z_{n+1})^l-(\log (m_{n+1}Z_n))^l|;Z_{n+1}>0,Z_n>0]+Ce^{-\beta n},
	\eeqlb
	where we use  Assumption (C), Lemmas \ref{pzn=0}, \ref{logzn} and \ref{logdeltan} in the last inequality.  Since $a^k-b^k=(a-b)(a^{k-1}+a^{k-2}b+\cdots+b^{k-1})$, by Lemma \ref{logdeltan} we have
	\beqnn
	&&\mathbb{E}_k[|(\log Z_{n+1})^l-(\log (m_{n+1}Z_n))^l|;Z_{n+1}>0,Z_n>0]\crcr
	&\leq &\sum_{j=0}^{l-1}\mathbb{E}_k[(\log Z_{n+1})^j|\log(m_{n+1}Z_n)|^{l-1-j}|\log\Delta_n|;Z_{n+1}>0,Z_n>0]\leq Ce^{-\beta n}
	\eeqnn
	with some $C,\beta>0$.
	Combining this with (\ref{jnk}) yields (\ref{key2}). Hence the induction argument is complete and we finish the proof of Lemma.
\end{proof}

\textbf{Proof of Theorem \ref{CLT}:}

Lemma \ref{key} implies that the sequence of functions $\phi_{k,n}(s)=\frac{\mathbb{E}_k[Z_n^{\mathrm{i}s};Z_n>0]}{\lambda(s)^n}$ converges uniformly on some interval $I_\eta$ to a continuous function $\phi$ and we also have $\phi(0)=1$ by Lemma \ref{pzn=0}. Therefore for every $s\in\mathbb{R}$,
\beqnn
\mathbb{E}_k\bigg[e^{\mathrm{i}s\frac{\log Z_n-n\mu}{\sqrt{n}\sigma}};Z_n>0\bigg]=\phi_{k,n}\bigg(\frac{s}{\sqrt{n}\sigma}\bigg)\lambda^n\bigg(\frac{s}{\sqrt{n}\sigma}\bigg)e^{-\mathrm{i}s\sqrt{n}\mu/\sigma}\rightarrow e^{-s^2/2} \quad \text{as} \quad n\rightarrow\infty,
\eeqnn
where we use the classical central limit theorem for the random walk $S_n$, i.e.
\beqnn
\lambda^n(\frac{s}{\sqrt{n}\sigma})e^{-\mathrm{i}s\sqrt{n}\mu/\sigma}=\mathbb{E}\bigg[e^{\mathrm{i}s\frac{S_n-n\mu}{\sqrt{n}\sigma}}\bigg]\rightarrow e^{-s^2/2}\quad \text{as} \quad n\rightarrow\infty
\eeqnn
for every $s\in\mathbb{R}$, since $\mathbb{E}(\log m_1)^2<\infty$. 
Thus we finish the proof.\qed

\textbf{Proof of Theorems \ref{edgeworth} and \ref{renewal}:} By Lemma~\ref{key}, there exist $\beta>0$ and a neighborhood $I_\eta$,   in which we can expand $\phi_{k,n}(s)$ as
$$
\phi_{k,n}(s)=\sum_{l=0}^{r-1}\frac{\phi_{k,n}^{(l)}(0)s^l}{l!}+O(s^r)=1+\sum_{l=0}^{r-1}\frac{\phi^{(l)}(0)s^l}{l!}+O(s^r)+o(se^{-\beta n}),\quad s\in I_\eta.
$$
Using similar methods with \cite{2022dar} and replacing the $\mathbb{E}_\mathcal{S}$ by $\mathbb{E}[\cdot;Z_n>0]$ as in the proof of Theorem \ref{CLT} we get the desired results.\qed

\noindent\bf{\footnotesize Acknowledgements}\quad\rm{\footnotesize
This work is supported in part by the National
Nature Science Foundation of China (Grant No.~12271043), the National Key Research and Development Program of China (No.~2020YFA0712900).\\[4mm]

%%%%%%%%%%%%%%%%%%%%%%%%%%%
	% Assume $\mathbb{P}\{f(\{0\})=1\}=0$.
	
	%%%%%%%%%%%%%%%%%%%%%%%%%%%%%%%%

\end{document}